\documentclass[12pt]{amsart}
\usepackage{amsmath,amsthm,latexsym,amscd,amsbsy,amssymb,amsfonts,fleqn,leqno,
euscript, pb-diagram}

\let\SavedRightarrow=\Rightarrow
\usepackage{marvosym}
\let\Rightarrow=\SavedRightarrow

\newcommand{\Tee }{\mathcal T}

\newcommand{\cl}{\operatorname{cl}}
\renewcommand{\int}{\operatorname{Int}}

\newcommand{\reg}{\mathrm{e}}

\newtheorem{thm}{Theorem}[section]
\newtheorem{pro}[thm]{Proposition}
\newtheorem{lem}[thm]{Lemma}
\newtheorem{que}[thm]{Question}
\newtheorem{cor}[thm]{Corollary}

\newtheorem*{MainTheorem1}{Theorem \ref{3.3}}
\begin{document}

\title{Skeletally  Dugundji spaces}

\author{A. Kucharski}
\address{Institute of Mathematics, University of Silesia, ul. Bankowa 14, 40-007 Katowice}
\email{akuchar@ux2.math.us.edu.pl}

\author{Sz. Plewik}
\address{Institute of Mathematics, University of Silesia, ul. Bankowa 14, 40-007 Katowice}
\email{plewik@math.us.edu.pl}

\author{V. Valov}
\address{Department of Computer Science and Mathematics, Nipissing University,
100 College Drive, P.O. Box 5002, North Bay, ON, P1B 8L7, Canada}
\email{veskov@nipissingu.ca}
\thanks{The third author was supported by NSERC Grant 261914-08}




\begin{abstract}
We introduce and investigate the class of skeletally Dugundji spaces as a skeletal analogue of Dugundji space. 
The main result states that the following conditions are equivalent for a given space $X$: (i) $X$ is skeletally Dugundji;
(ii) Every compactification of $X$ is co-absolute to a Dugundji space;
(iii) Every $C^*$-embedding of the absolute $p(X)$ in another space is strongly $\pi$-regular;
(iv) $X$ has a multiplicative lattice in the sense of Shchepin \cite{s76} consisting of skeletal maps.
\end{abstract}

\maketitle 
\markboth{}{Skeletally  Dugundji spaces}

\section{Introduction}

In this paper we introduce a class of skeletally Dugundji spaces as a skeletal analogue of Dugundji spaces \cite{p}. 
The paper can be considered as a continuation of \cite{kp7},  \cite{kp8}, \cite{kp9} and \cite{vv}, where it was shown that 
I-favorable  spaces \cite{dkz} coincide with the class of skeletally generated spaces \cite{vv}. The last class is a
skeletal counterpart of $\kappa$-metrizable compacta \cite{s81}.

Recall that a map $f:X \to Y$ is called
\textit{skeletal} \cite{mr} (resp., {\em semi-open}) if the set $\int_Y\cl_Y f(U)$ (resp., $\int_Yf(U)$) is non-empty, for
any $U\in \Tee_X$. Obviously, every semi-open map is skeletal, and both notions are equivalent for closed maps.
Our definition of skeletally Dugundji spaces is similar to the spectral characterization of Dugundji spaces obtained by
Haydon \cite{ha}. We say that a space $X$ is {\em skeletally Dugundji} if there exists
an inverse system
$\displaystyle S=\{X_\alpha, p^{\beta}_\alpha, \alpha<\beta<\tau\}$ with surjective skeletal bonding maps, where $\tau$ is identified
with the first
ordinal $\omega(\tau)$ of cardinality $\tau$, satisfying the following conditions: (i) $X_0$ is a separable metric space
and all maps $p^{\alpha+1}_\alpha$ have metrizable kernels (i.e., there exists a separable metric space $M_\alpha$ such that
$X_{\alpha+1}$ is embedded in $X_{\alpha}\times M_\alpha$ and $p^{\alpha+1}_\alpha$ coincides with the restriction $\pi|X_{\alpha+1}$ of
the projection $\pi\colon X_{\alpha}\times M_\alpha\to X_{\alpha}$);
(ii) for any limit ordinal $\gamma<\tau$ the space $X_\gamma$ is a (dense)
subset of
$\displaystyle\lim_\leftarrow\{X_{\alpha},p^\beta_\alpha, \alpha<\beta<\gamma\}$;
(iii) $X$ is embedded in $\displaystyle\lim_\leftarrow
S$ such that $p_\alpha(X)=X_\alpha$ for each $\alpha$, where
$p_\alpha\colon\displaystyle\lim_\leftarrow S\to X_\alpha$ is the
$\alpha$-th limit projection; (iv) for every bounded continuous
real-valued function $f$ on $\displaystyle\lim_\leftarrow S$ there exists $\alpha\in A$ such that $p_\alpha\prec f$ (the last relation means that  there exists a continuous function $g$ on  $X_\alpha$ with $f=g\circ p_\alpha$).
If the inverse system $S$ and $X$ satisfy conditions (ii) and (iii) $X$ is said to be the {\em almost limit} of $S$, notation $X=\mathrm{a}-\displaystyle\lim_\leftarrow S$.  We also say that an inverse system $S$ is factorizing if it satisfies condition $(iv)$.

In the paper we provide different characterizations of skeletally Dugundji spaces. One of our starting points was the result of Shapiro \cite{sh}
that a compactum $X$ is co-absolute to a Dugundji space (i.e., the absolute of $X$ is the absolute of a Dugundji space) if and only if $X$ is the limit space of a continuous inverse system
$\displaystyle S=\{X_\alpha, p^{\beta}_\alpha, \alpha<\beta<\tau\}$ satisfying all conditions from the definition of skeletally Dugundji spaces, excepts that the maps $p^{\alpha+1}_\alpha$ don't necessarily have metrizable kernels but have countable $\pi$-weight (see Section 3 for this notion). Let us note that necessity of the above result was announced in \cite[Theorem 3]{sh} without a proof.
We establish in Section 2 that any space co-absolute to a space with a multiplicative in the sense of Shchepin \cite{s76} lattice of open maps
has a multiplicative lattice of skeletal maps. This extends the necessity of Shapiro's result, mentioned above. 
Some properties of skeletally Dugundji spaces are provided in Section 3. The following result is the main theorem in this section.   
\begin{MainTheorem1}
For any space $X$ the following are equivalent:
\begin{itemize}
\item[(i)] $X$ is skeletally Dugundji;
\item[(ii)] Every compactification of $X$ is co-absolute to a Dugundji space;
\item[(iii)] Every $C^*$-embedding of the absolute $p(X)$ in another space is strongly $\pi$-regular;
\item[(iv)] $X$ has a multiplicative lattice of skeletal maps.
\end{itemize}
\end{MainTheorem1}

Here, we say that a subspace $X$ of a space $Y$ is strongly $\pi$-regularly embedded in $Y$ if there exists a function $\mathrm{e}\colon\Tee_X\to\Tee_Y$ between the topologies of $X$ and $Y$ such that:
\begin{itemize}
\item $(e1)$ $\mathrm{e}(\varnothing)=\varnothing$ and $\mathrm{e}(U)\cap X$ is a dense subset of $U$;
\item $(e2)$ $\mathrm{e}(U\cap V)=e(U)\cap\mathrm{e}(V)$ for any $U,V\in\Tee_X$.
\end{itemize}
Such a function was considered in \cite{sh1} under the name $\pi$-regular operator. It follows from Corollary 3.2 that if
every embedding of a compactum $X$ in another space is strongly $\pi$-regular, then $X$ is skeletally Dugundji. A positive answer to the
following question would provide an external characterization of skeletally Dugundji spaces similar to the characterization of Dugundji spaces
given by Shirokov in \cite{shi}. 

\begin{que}
Is any embedding of a skeletally Dugundji compactum in a Tychonoff cube strongly $\pi$-regular?
\end{que}
Another question arises from Corollary 3.6 that if $X$ is a skeletally
Dugundji space, then the absolute $p(\beta X)$ of $\beta X$ is a retract of any
extremally disconnected space in which $p(\beta X)$ is embedded.

\begin{que}
 Let X be a compact space such that its absolute $p(X)$
is a retract of any extremally disconnected space containing $p(X).$ Is
X skeletally Dugundji?
\end{que}

When $X$ is 0-dimensional, this questions is equivalent to the open
problem after Observation 5.3.10 from \cite{hsh}. According to \cite{sh2}, Question
1.2 has a positive answer if the weight of X is $\leq \omega_1.$

All spaces in this paper are Tychonoff and the single-valued maps are continuous.

\section{Spaces co-absolute with $AE(0)$-spaces}

Haydon \cite{ha} established that the class of compact absolute extensors for zero-dimensional spaces (br., $AE(0)$-spaces) coincides with the class of Dugundji spaces and any compactum $X$ belongs
to that class iff $X$ can be represented as the limit space of a continuous inverse system
$\displaystyle S=\{X_\alpha, p^{\beta}_\alpha, \alpha<\beta< w(X)\}$ with open bonding maps such that each map $p^{\alpha+1}_\alpha$ has a metrizable kernel, see \cite{ha}.  Dugundji spaces can be also characterized as the compact spaces $X$ possessing a multiplicative
lattice in the sense of Shchepin \cite{s76} consisting of open maps. This means that there exists a family $\Psi$ of open maps
with domain $X$ such that:
\begin{itemize}
\item[(L1)] For any map $f\colon X\to f(X)$ there exists $\phi\in\Psi$ with $\phi\prec f$ and $w(\phi(X))\leq w(f(X))$;
\item[(L2)] $\Psi$ is multiplicative, i.e., if $\{\phi_\alpha:\alpha\in A\}\subset\Psi$, then the diagonal product
$\triangle\{\phi_\alpha:\alpha\in A\}$ belongs to $\Psi$.
\end{itemize}
Let us note that a general definition of $AE(0)$ in the class of all Tychonoff spaces was introduced by Chigogidze \cite{ch}. By \cite[Theorem 1]{v2}, for every $C$-embedding of an $AE(0)$-space $X$ in $\mathbb R^A$ there exists an upper semi-continuous compact-valued map 
$r\colon\mathbb R^A\to X$  with $r(x)=\{x\}$ for all $x\in X$. Then, following the terminology of \cite[Lemmas 3-6]{v2}, all restrictions $\pi_B|X$, where $B\subset A$ is $r$-admissible, are open and form a multiplicative lattice. Therefore, every $AE(0)$-space possesses a multiplicative lattice of open maps.

\begin{pro}\label{og}
Let $X$ be a $C$-embedded in $\mathbb R^\Gamma$ for some $\Gamma$ and $X$ has a multiplicative lattice $\Psi.$ Then the family $\mathcal A=\{B\subset\Gamma: \phi_B\in\Psi\hbox{~}\}$, where $\phi_B=\pi_B|X\colon X\to\pi_B(X)$ is the restriction of the projection $\pi_B:\mathbb R^\Gamma\to\mathbb R^B$, has the following properties:
\begin{itemize}
\item[(i)] $\mathcal A$ is additive, i.e., the union of any elements from $\mathcal A$ is also from $\mathcal A$;
\item[(ii)] Every $C\subset\Gamma$ is contained in some  $B\in\mathcal A$ with $|B|\leq |C|\cdot\aleph_0$.
\end{itemize}
\end{pro}

\begin{proof}
Let $\Psi$ be a multiplicative lattice on $X$ consisting of quotient maps.
Suppose $B=\bigcup B_\alpha$
with $B_\alpha\in\mathcal A$ for all $\alpha$. Then
 $\phi_B=\triangle\phi_{B_\alpha}$ because $\triangle\phi_{B_\alpha}\in\Psi$ is quotient. So, $B\in\mathcal A$.  Assume $C\subset\Gamma$ is  infinite of cardinality $|C|=\tau$. We construct by
induction an increasing sequence $\{B(k)\}\subset\Gamma$ and a sequence $\{\phi_k\}\subset\Psi$ such that $B(1)=C$, $|B(k)|=\tau$,
$w(\phi_k(X))\leq\tau$ and
$\phi_{B(k+1)}\prec\phi_k\prec\phi_{B(k)}$ for all $k$. Suppose the construction is done up to level $k$ for some $k\geq 1$.  We consider each
$\phi_k(X)$ as a subspace of $\mathbb R^{\tau}$. Since $X$ is $C$-embedded
in $\mathbb R^\Gamma$, there exists a map $g_k\colon\mathbb R^\Gamma\to\mathbb R^{\tau}$ extending $\phi_k$. Then $g_k$ depends on $\tau$
many coordinates of $\mathbb R^\Gamma$, so we can find a set $B(k+1)\subset\Gamma$ of cardinality $\tau$ containing $B(k)$ such that $\pi_{B(k+1)}\prec g_k$.
Consequently, $\phi_{B(k+1)}\prec \phi_k$. Next, by condition $(\mathrm{L1})$, there exists $\phi_{k+1}\in\Psi$ with $\phi_{k+1}\prec \phi_{B(k+1)}$ and
$w(\phi_{k+1}(X))\leq\tau$. This completes the construction. Finally, let $B=\bigcup_{k=1}^{\infty}B(k)$ and
$\phi=\triangle_{k=1}^{\infty}\phi_k$. Obviously, $|B|=\tau$ and $\phi_B=\phi\in\Psi$. Hence, $C\subset B\in\mathcal A$.
\end{proof}

Shapiro \cite[Theorem 3]{sh} stated without a proof that if a compactum $X$ is co-absolute to a Dugundji space, then $X$ is the limit space of a continuous inverse system $\displaystyle S=\{X_\alpha, p^{\beta}_\alpha, \alpha<\beta< w(X)\}$
such that $X_0$ is a point, the bonding maps are skeletal and each $p^{\alpha+1}_\alpha$ has a countable $\pi$-weight (see Section 3 for this notion). Next theorem is a generalization of Shapiro's result (recall that any Dugundji space is an $AE(0)$, and hence has a multiplicative lattice of open maps).

\begin{thm}
Let $X$ be a space with a multiplicative lattice of open maps. Then every space co-absolute with $X$ has a
multiplicative lattice of semi-open maps.
\end{thm}

\begin{proof}
Suppose $Y$ is co-absolute with $X$ and $Z$ is their common absolute. So, there exist
irreducible perfect maps $\theta_X\colon Z\to X$ and $\theta_Y\colon Z\to Y$.  Consider the set-valued maps
$r_X\colon X\to Y$ and $r_Y\colon Y\to X$ defined by
$r_X(x)=\theta_Y(\theta_X^{-1}(x))$, $x\in X$, $r_Y(y)=\theta_X(\theta_Y^{-1}(y))$, $y\in Y$.
Since $\theta_X$ is irreducible, for every open $V\subset Y$ the set $r_X^{\sharp}(V)=\{x\in X:r_X(x)\subset V\}$ is non-empty and open in $X$, and
$r_X^{\sharp}(V)=\theta_X^{\sharp}(\theta_Y^{-1}(V))$, where $\theta_X^{\sharp}(W)=\{x\in X:\theta_X^{-1}(x)\subset W\}$, $W\subset Z$.
 Hence, $r_X$ is upper semi-continuous and compact valued. Similarly, $r_Y$ is also
upper semi-continuous and compact valued.

Next claim follows from the facts that both $\theta_X$ and $\theta_Y$ are closed irreducible maps, and $Z$ is extremally disconnected.

\textit{Claim $1.$ For every open $V\subset Y$ the set $\theta_X^{\sharp}(\overline{\theta_Y^{-1}(V)})$ is regularly open and
$\overline{r_X^{\sharp}(V)}=\overline{r_X^{-1}(V)}=\overline{\theta_X^{\sharp}(\overline{\theta_Y^{-1}(V)})}$.}

Consider the disjoint union $X\bigoplus Y$ of $X$ and $Y$ as a $C$-embedded subspace in $\mathbb R^\Gamma$ for some $\Gamma$.
For every $B\subset\Gamma$ let $\phi_B=\pi_B|X$, $p_B=\pi_B|Y$, $X_B=\phi_B(X)$ and
$Y_B=p_B(Y)$. If $B\subset C\subset\Gamma$, there exists natural maps $\phi^C_B\colon X_C\to X_B$ and $p^C_B\colon Y_C\to Y_B$.
Let $\Psi$ be a multiplicative lattice for $X$ consisting of open maps.
Since $X$ is also
$C$-embedded in $\mathbb R^\Gamma$, the family $\mathcal A=\{B\subset\Gamma: \phi_B\in\Psi\hbox{~}\}$ satisfies
conditions $(i) - (ii)$ from Proposition~\ref{og}.

\textit{Claim $2.$ Let $C\in\mathcal A$ be a set of cardinality $\tau$ and $\{V_\alpha\}$ an open family in $Y_C$ of the same
cardinality. Then there exists $B\in\mathcal A$ containing $C$ and
a family $\{G_\alpha\}$ of open subsets of $\phi_B(X)$ such that $|B|=\tau$ and
$r_X^{\sharp}(p_B^{-1}(V_\alpha))$ is a dense subset of $\phi_B^{-1}(G_\alpha)$ for each $\alpha$.}

Because every disjoint open family in $X$ is at most countable \cite{va1}, any family of open subsets of $X$  contains a dense countable subfamily.
For every $\alpha$ choose a cover $\gamma_\alpha$ of $r_X^{\sharp}(p_B^{-1}(V_\alpha))$
consisting of open subsets of $X$ of the form $U\cap X$, where $U$ is a standard open set in $\mathbb R^\Gamma$. Then there exists a dense countable subfamily $\omega_\alpha$ of $\gamma_\alpha$. Since each element of $\omega_\alpha$ depends on finitely many coordinates, by Proposition~\ref{og}, we can
 find a set $B\in\mathcal A$ containing $C$ of cardinality $|B|=\tau$ such that
$\phi_{B}(W)$ is open in $\phi_{B}(X)$ and $\phi_{B}^{-1}(\phi_{B}(W))=W$ for every $W\in\omega_\alpha$ and every $\alpha$.  Then each $O_\alpha=\phi_B(W_\alpha)$ is open in $\phi_B(X)$ and $\phi_B^{-1}(O_\alpha)=W_\alpha$, where
$W_\alpha=\bigcup\{W:W\in\omega_\alpha\}$. Because $\phi_B$ is open, we have $\phi_B^{-1}(\overline{O_\alpha})=\overline{W_\alpha}=\overline{r_X^{\sharp}(p_B^{-1}(V_\alpha))}$
and  $\phi_B(r_X^{\sharp}(p_B^{-1}(V_\alpha)))\subset \rm{Int}(\overline{O_\alpha})=G_\alpha$. Hence,
$r_X^{\sharp}(p_B^{-1}(V_\alpha))\subset\phi_B^{-1}(G_\alpha)\subset \overline{r_X^{\sharp}(p_B^{-1}(V_\alpha))}$ for every $\alpha$.
This completes the proof of Claim 2.

 For any $B\subset\Gamma$  let $\Omega_B$ and
$\Lambda_B$ be bases for $X_B$ and $Y_B$, respectively,  having cardinality $\leq |B|$ such that $\Omega_B$ consists of open sets of the form
 $U\cap X_B$ and $\Lambda_B$ consists of all finite unions of sets of the form $V\cap Y_B$, where $U$ and $V$ are standard open sets in $\mathbb R^{B}$.
Denote by $\Sigma$ the family of all $B\subset\Gamma$ satisfying the following conditions:
\begin{itemize}
\item[(1)] $B\in\mathcal A$;
\item[(2)] For every $V\in\Lambda_B$ there exists an open set $G_V\subset\phi_B(X)$ such that
$r_X^{\sharp}(p_B^{-1}(V))$ is a dense subset of $\phi_B^{-1}(G_V)$;
\item[(3)] For any $U\in\Omega_B$ there exists an open set $V_U\subset Y_B$ with
$p_B^{-1}(V_U)\subset r_Y^{\sharp}(\phi_B^{-1}(U))$.
\end{itemize}

\textit{Claim $3.$ $\Sigma$ is additive.}

First, let us show that $\Sigma$ is finitely additive. It suffices to prove that if $B(1),B(2)$ belong to $\Sigma$, then
so is $B=B(1)\cup B(2)$. Because $\mathcal A$ is additive, we need to check that $B$ satisfies conditions $(2)$ and $(3)$.

Suppose $V\subset Y_B$ is open and $V=V_1\cup V_2\cup...\cup V_m$,
 where each $V_i$ is of the form $W_i\cap Y_B$ with $W_i$ being a standard open subset of $\mathbb R^{B}$. So, for every $i$ we have
 $W_i=\prod\{W_i(\alpha):\alpha\in B\}$ such that all $W_i(\alpha)$, $\alpha\in B$, are open intervals and the set
 $k(V_i)=\{\alpha: W_i(\alpha)\neq\mathbb R\}$ is finite. Then the family $\{V_1,...,V_m\}$ is the union of the following families:
 $\gamma_j=\{V_i: k(V_i)\subset B(j)\}$, $j=1,2$, and $\gamma_{1,2}=\{V_i: k(V_i)\backslash B(j)\neq\varnothing, j=1,2\}$.
 Let $O_j^*=\bigcup\{p^B_{B(j)}(V_i):V_i\in\gamma_j\}$ and $O_j=\bigcup\{V_i:V_i\in\gamma_j\}$, $j=1,2$. Then, according to (2),
 $r_X^{\sharp}(p_{B(j)}^{-1}(O_j^*))$ is a dense subset of $\phi_{B(j)}^{-1}(G_j^*)$ for some open $G_j^*\subset X_{B(j)}$, $j=1,2$.
 Since $p_{B(j)}^{-1}(O_j^*)=p_{B}^{-1}(O_j)$, for every $j=1,2$ we have
 $$r_X^{\sharp}(p_B^{-1}(O_j))\hbox{~}\mbox{is dense in}\hbox{~}\phi_{B}^{-1}(G_j)\hbox{~}\mbox{with}\hbox{~} G_j=(\phi^B_{B(j)})^{-1}(G_j^*). \leqno{(4)}$$
If $V_i\in\gamma_{1,2}$, then $V_i=V_i(1)\cap V_i(2)$ with $V_i(j)\in\gamma_j$, $j=1,2$. Hence, for each $j$ there exists an open set $G_{V_i(j)}\subset X_{B(j)}$ such that
$\phi_{B(j)}^{-1}(G_{V_i(j)})$ contains $r_X^{\sharp}(p_{B(j)}^{-1}(V_i(j)^*)$ as a dense subset, where $V_i(j)^*=p^B_{B(j)}(V_i(j))$. Because
$p_B^{-1}(V_i)=p_{B(1)}^{-1}(V_i(1)^*)\cap p_{B(2)}^{-1}(V_i(2)^*)$, the set $r_X^{\sharp}(p_{B}^{-1}(V_i))$ (being the intersection of the open sets  $r_X^{\sharp}(p_{B(1)}^{-1}(V_i(1)^*))$ and $r_X^{\sharp}(p_{B(2)}^{-1}(V_i(2)^*))$) is dense in $\phi_B^{-1}(G_{V_i})$, where
$G_{V_i}=(\phi^B_{B(1)})^{-1}(G_{V_i(1)})\cap (\phi^B_{B(2)})^{-1}(G_{V_i(2)})$. Therefore, we have finitely many open sets $G_W\subset X_B$ such that $$r_X^{\sharp}(p_{B}^{-1}(W))\hbox{~}\mbox{is dense in}\hbox{~}\phi_B^{-1}(G_W), W\in\gamma_{1,2}. \leqno{(5)}$$

Obviously, $V=\bigcup\{W:W\in\gamma_{1,2}\}\cup O_1\cup O_2$, and let $G(V)=\bigcup\{G_W:W\in\gamma_{1,2}\}\cup G_1\cup G_2$. It follows from $(4)$ and (5) that the set
$$L=\bigcup\{r_X^{\sharp}(p_{B}^{-1}(W)):W\in\gamma_{1,2}\}\cup r_X^{\sharp}(p_B^{-1}(O_1))\cup r_X^{\sharp}(p_B^{-1}(O_2))$$ is
dense in $\phi_B^{-1}(G(V))$. On the other hand, by Claim 1,
$$\overline{L}=\bigcup\{\overline{r_X^{-1}(p_{B}^{-1}(W))}:W\in\gamma_{1,2}\}\cup\overline{r_X^{-1}(p_B^{-1}(O_1))}\cup\overline{ r_X^{-1}(p_B^{-1}(O_2))}.$$ Therefore, $\overline{L}=\overline{r_X^{-1}(p_B^{-1}(V))}=\overline{\phi_B^{-1}(G(V))}$. Let
$G_V=\mathrm{Int}(\overline{G(V)})$. Because $\phi_B$ is an open map, $\phi_B^{-1}(G_V)=\mathrm{Int}(\overline{r_X^{-1}(p_B^{-1}(V))})$. Finally,
since $r_X^{\sharp}(p_{B}^{-1}(V))$ is open in $X$ and it is dense in $\overline{r_X^{-1}(p_B^{-1}(V))}$ (see Claim 1),
$r_X^{\sharp}(p_{B}^{-1}(V))$ is a dense subset of $\phi_B^{-1}(G_V)$. Thus, $B$ satisfies condition $(2)$.

To show that $B$ satisfies condition $(3)$, let $U\in\Omega_B$. Then $U=\phi^B_{B(1)}(U(1))\cap\phi^B_{B(2)}(U(2))$ with $U(i)\in\Omega_{B(i)}$, $i=1,2$. So, there are
open sets $V_{U(i)}\subset Y_{B(i)}$ such that
$p_{B(i)}^{-1}(V_{U(i)})\subset r_Y^{\sharp}(\phi_{B(i)}^{-1}(U(i)))$, $i=1,2$. Let
$V_U=(p^B_{B(1)})^{-1}(V_{U(1)})\cap (p^B_{B(2)})^{-1}(V_{U(2)})$. Then, $p_B^{-1}(V_U)=p_{B(1)}^{-1}(V_{U(1)})\cap p_{B(2)}^{-1}(V_{U(2)})$
and $\phi_B^{-1}(U)=\phi_{B(1)}^{-1}(U(1))\cap\phi_{B(2)}^{-1}(U(2))$. Consequently,
$p_{B}^{-1}(V_{U})\subset r_Y^{\sharp}(\phi_{B}^{-1}(U))$.
Hence, $B\in\Sigma$.

Suppose now that $B=\bigcup B_\alpha$ is the union of infinitely many $B_\alpha\in\Sigma$, and let $V\in\Lambda_B$.
Then there exists a set $C\subset B$ which is a union of  finitely many $\alpha_i$, $i=1,..,k$, such that
$(p^B_C)^{-1}(p^B_C(V))=V$ and $V^*=p^B_C(V)\in\Lambda_C$. Since $C\in\Sigma$, there exists an
open set $G_V^*$ in $X_C$ with $r_X^{\sharp}(p_C^{-1}(V^*))$ being a dense subset of $\phi_C^{-1}(G_V^*)$. Obviously,
$p_C^{-1}(V^*)=p_B^{-1}(V)$ and $\phi_C^{-1}(G_V^*)=\phi_B^{-1}(G_V)$, where $G_V=(\phi_C^B)^{-1}(G_V^*)$. Consequently,
$r_X^{\sharp}(p_B^{-1}(V))$ is a dense subset of $\phi_B^{-1}(G_V)$. Hence, $B$ satisfies $(2)$. Similarly, one can show that
$B$ satisfies also condition $(3)$. Therefore, $B\in\Sigma$. This complete the proof of Claim 3.

\textit{Claim $4.$ Every infinite $C\in\mathcal A$ of cardinality $\tau$ is contained in a set $B\in\Sigma$ with $|B|=\tau$.}

We construct by induction sets $B(k)\in\mathcal A$, open subsets $\{G_V:V\in\Lambda_{B(2k)}\}$ of $\phi_{B(2k+1)}(X)$ and open subsets $\{V_U:U\in\Omega_{B(2k+1)}\}$ of $Y_{B(2k+2)}$  such that for every $k\geq 0$ we have:
\begin{itemize}
\item[(i)] $B(0)=C$, $|B(k)|=\tau$ and  $B(k)\subset B(k+1)$;
\item[(ii)] $r_X^{\sharp}(p_{B(2k)}^{-1}(V))$ is a dense subset of $\phi_{B(k+1)}^{-1}(G_V)$ for every
$V\in\Lambda_{B(2k)}$;
\item[(iii)] $p_{B(2k+2)}^{-1}(V_U)\subset r_Y^{\sharp}(\phi_{B(2k+1)}^{-1}(U))$ for every $U\in\Omega_{B(2k+1)}$.
\end{itemize}

Suppose the construction is done up to level $2k$. Since the family $\{p_{B(2k)}^{-1}(V):V\in\Lambda_{B(2k)}\}$ is of cardinality $\leq\tau$, by Claim 2, there
exists a set $B(2k+1)\in\mathcal A$ of cardinality $\tau$ containing $B(2k)$ and open sets $\{G_V:V\in\Lambda_{B(2k)}\}$ in $\phi_{B(k+1)}(X)$ satisfying item $(ii)$.
Because each $r_Y^{\sharp}(\phi_{B(2k+1)}^{-1}(U))$, $U\in\Omega_{B(2k+1)}$, is open in $Y$, we can find a set $B(2k+2)\in\mathcal A$ containing $B(2k+1)$ such that $|B(k+2)|=\tau$ and
$p_{B(2k+2)}(Y)$ contains an open family $\{V_U:U\in\Omega_{B(2k+1)}\}$ satisfying item $(iii)$. This complete the inductive step.
Obviously,
$B=\bigcup_{k=1}^{\infty}B(k)\in\mathcal A$ and $|B|=\tau$. Repeating the arguments from the proof of Claim 3, one can show that $B$ satisfies conditions $(2)$ and $(3)$. So, $B\in\Sigma$, which completes the proof of Claim 4.

\textit{Claim $5.$ For any $B\in\Sigma$ the map $p_B\colon Y\to Y_B$ is semi-open.}

First, let us show that if $B\in\Sigma$ and $\phi_B(x_1)=\phi_B(x_2)$ for some $x_1, x_2\in X$, then $p_B(r_X(x_1))=p_B(r_X(x_2))$. Indeed,
suppose $p_B(r_X(x_2))\subset V$, where $V\subset Y_B$ is open. Since $p_B(r_X(x_2))$ is compact and $\Lambda_B$ is finitely additive, we can
assume that $V\in\Lambda_B$. Then $x_2\in r_X^{\sharp}(p_B^{-1}(V))$. By condition $(2)$, $r_X^{\sharp}(p_B^{-1}(V))$ is a dense subset of
$\phi_B^{-1}(G_V)$ for some open set $G_V$ in $X_B$, so
$x_1\in\phi_B^{-1}(x_2)\subset\phi_B^{-1}(G_V)$. On the other hand, according to Claim 1, we have
$\theta_X^{\sharp}(\overline{\theta_Y^{-1}(p_B^{-1}(V))})=\mathrm{Int}(\overline{r_X^{\sharp}(p_B^{-1}(V))})$. Hence,
$\theta_X^{\sharp}(\overline{\theta_Y^{-1}(p_B^{-1}(V))})$ contains the set $\phi_B^{-1}(G_V)$. Thus,
$r_X(x_1)\subset\theta_Y(\overline{\theta_Y^{-1}(p_B^{-1}(V))})=\overline{p_B^{-1}(V)}$. Finally, we obtain $p_B(r_X(x_1))\subset\overline{V}$,
which implies $p_B(r_X(x_1))\subset p_B(r_X(x_2))$. Similarly, $p_B(r_X(x_2))\subset p_B(r_X(x_1))$.

To show that $p_B$ is semi-open, let $W\subset Y$ be open. Then $\phi_B(r_X^{\sharp}(W))$ is open in $\phi_B(X)$. According to $(3)$, there exists $U\in\Omega_B$ and open $V_U\subset Y_B$ with
$U\subset\phi_B(r_X^{\sharp}(W))$ and $p_B^{-1}(V_U)\subset r_Y^{\sharp}(\phi_B^{-1}(U))$. The last inclusion implies $V_U\subset p_B(r_Y^{\sharp}(\phi_B^{-1}(U)))$. It is easily seen that $r_Y^{\sharp}(\phi_B^{-1}(U))\subset r_X(\phi_B^{-1}(U))$. So, we have the inclusions
$$V_U\subset p_B(r_Y^{\sharp}(\phi_B^{-1}(U)))\subset p_B(r_X(\phi_B^{-1}(U))).$$  Therefore, it suffices to show that
$p_B(r_X(\phi_B^{-1}(U)))$ is contained in $p_B(W)$. To this end, let $x\in\phi_B^{-1}(U)$. Then there exists $y\in r_X^{\sharp}(W)$ with $\phi_B(x)=\phi_B(y)$. Hence, $p_B(r_X(x))=p_B(r_X(y))\subset p_B(W)$. Thus,  $p_B(r_X(\phi_B^{-1}(U)))\subset p_B(W)$, which completes the
proof of Claim 5.

We can show now that $Y$ has a multiplicative lattice of semi-open maps. Indeed, let $\Psi_Y=\{p_B:B\in\Sigma\}$. According to the last claim,
$\Psi_Y$ consists of semi-open maps. If $\{p_{B(\alpha)}\}\subset\Psi_Y$, then it is easily seen that $\triangle p_{B(\alpha)}=p_B$, where
$B$ is the union of all $B(\alpha)$. Hence, $\triangle p_{B(\alpha)}\in\Psi_Y$. Finally, let $g\colon Y\to g(Y)$ be an arbitrary map with
$w(g(Y))=\tau$. Considering $g(Y)$ as a subspace of $\mathbb R^\tau$, we can extend $g$ to a map $h\colon\mathbb R^\Gamma\to\mathbb R^\tau$ (recall that $Y$ is $C$-embedded in $\mathbb R^\Gamma$). Then there exists a set $B\in\Sigma$ of cardinality $\tau$ such that $\pi_B\prec h$. Consequently,
$p_B\prec g$ and $w(p_B(Y))\leq\tau$.
\end{proof}

\begin{cor}
Every space $Y$ co-absolute with a space possessing a multiplicative lattice of open maps is skeletally Dugundji.
In particular, if $Y$ is co-absolute to an $AE(0)$-space, then $Y$ has a multiplicative lattice of semi-open maps and $Y$ is skeletally Dugundji.
\end{cor}

\begin{proof}
We are going to show that $Y=\mathrm{a}-\displaystyle\lim_\leftarrow S$, where  $\displaystyle S=\{Y_\alpha, p^{\beta}_\alpha, \alpha<\beta<\tau\}$ is a factorizing inverse system with skeletal maps such that $Y_0$ is a separable metric space  and $p^{\alpha+1}_\alpha$ has a metrizable kernel for each $\alpha$.
Let $\mathbb R^\Gamma$ and $\Sigma$ be as in the proof of Theorem 2.2. There is nothing to prove if $Y$ is second countable. So, let $w(Y)>\aleph_0$ and assume that $\Gamma$ is the set of all ordinals $\alpha<\omega(\tau)$.
For every $\alpha\in\Gamma$ there exists  a countable set $A(\alpha)\in\Sigma$ containing $\alpha$. Define $B(0)=A(0)$,
$B(\alpha)=\bigcup\{A(\beta):\beta<\alpha\}$ if $\alpha$ is a limit ordinal, and $B(\alpha)=\bigcup\{A(\beta):\beta\leq\alpha\}$ if $\alpha$ is isolated.
Then the inverse system $S=\{Y_{B(\alpha)}, p^{B(\beta)}_{B(\alpha)}, \alpha<\beta<\omega(\tau)\}$ consists of skeletal maps and  $p^{B(\alpha+1)}_{B(\alpha)}$ has a metrizable kernel for all $\alpha$.
Moreover, $Y$ as a dense subset of $\displaystyle\lim_\leftarrow
S$ with $Y=\mathrm{a}-\displaystyle\lim_\leftarrow S$.

The second part of the corollary follows from the fact that each $AE(0)$-space has a multiplicative
lattice of open maps.
\end{proof}

\begin{cor}
Let $Y$ be a Lindel\"{o}f $p$-space co-absolute  with a space possessing a multiplicative lattice of open maps. Then $Y$ has a multiplicative lattice
of perfect skeletal maps.
\end{cor}

\begin{proof}
Suppose $Y$ is co-absolute with a space $X$ possessing a multiplicative lattice of open maps.
Recall that a Lindel\"{o}f $p$-space \cite{a} is a space admitting a perfect map onto a separable metric space, and that this property is invariant under images and preimages of perfect maps. So, $X$ is also a a Lindel\"{o}f $p$-space. Then $X$ and $Y$ can be embedded as closed subsets of a product of the form
$M\times\mathbb I^\Gamma$ for some $\Gamma$, where $M$ is a second countable space. Considering in Proposition 2.1 such a product instead of $\mathbb R^\Gamma$, we obtain a family $\mathcal A=\{B\subset\Gamma: \phi_B\in\Psi\hbox{~}\}$, where $\phi_B=\pi_B|X\colon X\to\pi_B(X)$ is the restriction of the projection $\pi_B:M\times\mathbb I^\Gamma\to M\times\mathbb I^B$. Hence, all $\phi_B$, $B\in\mathcal A$ are perfect and open maps. Similarly, replacing $\mathbb R^\Gamma$ in Theorem 2.2 with $M\times\mathbb I^\Gamma$, we obtain the family $\Sigma\subset\mathcal A$ and that the perfect skeletal maps $p_B=\pi_B|Y:Y\to\pi_B(Y)$, $B\in\Sigma$, form a multiplicative lattice.
\end{proof}

\section{Skeletally Dugundji spaces}

Next proposition provides examples of skeletally Dugundji spaces.
\begin{pro} A space $X$ is skeletally Dugundji if it satisfies one of the following conditions:
\begin{itemize}
\item[(i)] $X$ has a multiplicative lattice of skeletal maps;
\item[(ii)] $X$ is strongly $\pi$-regularly  $C^*$-embedded subset of a space with a multiplicative lattice of open maps.
\end{itemize}
\end{pro}

\begin{proof}
Suppose $X$ has a multiplicative lattice $\Psi$ of skeletal maps. For every continuous function $f\in C(X)$  fix a map $\phi_f\in\Psi$ such that $\phi_f\prec f$ and $\phi_f(X)$ is second countable. We assume that $C(X)=\{f_\alpha:\alpha<\omega(\tau)\}$ for some cardinal $\tau$. Let
$\varphi_0=\phi_{f_0}$, $\varphi_\alpha$ is the diagonal product $\triangle_{\beta<\alpha}\phi_{f_\beta}$ if $\alpha$ is a limit ordinal, and
$\varphi_\alpha=\triangle_{\beta\leq\alpha}\phi_{f_\beta}$ if $\alpha$ is isolated. Define $X_\alpha=\varphi_\alpha(X)$ and
$\varphi_{\beta}^\alpha:X_\alpha\to X_\beta$, $\beta<\alpha$, to be the natural projection. Since all $\varphi_\alpha\in\Psi$, the inverse system
$\displaystyle S=\{X_\alpha, \varphi_{\beta}^\alpha: \beta<\alpha<\omega(\tau)\}$ consists of skeletal maps and $X$ is the almost limit of $S$. It follows from the definition of the maps $\varphi_\alpha$ that for every $f\in C(X)$ there exists $\alpha$ with $\varphi_\alpha\prec f$. So, $S$ is factorizing. This provides the proof of item (i).

Let $M$ be a space with a multiplicative lattice $\Phi$ of open maps and $X$ be strongly $\pi$-regularly  $C^*$-embedded in $M$.
Let $\mathrm{e}\colon\Tee_X\to\Tee_M$ be a strongly $\pi$-regular operator and $M_\phi=\phi(M)$ for any $\phi\in\Phi$.
We say that a map $\phi\in\Phi$ is $\mathrm{e}$-admissible if
$$\phi^{-1}(\phi(\overline{\mathrm{e}(\phi^{-1}(U)\cap X)}))=\overline{\mathrm{e}(\phi^{-1}(U)\cap X)}\hbox{~}\mbox{for all}\hbox{~}U\in\mathcal B_\phi,\leqno{(6)}$$
where $\mathcal B_\phi$ is an open
base for $M_\phi$ of cardinality $w(M_\phi)$. We are going
to show that the family $\Phi_X=\{\phi|X:\phi\hbox{~}\mbox{is $\mathrm{e}$-admissible}\hbox{~}\}$ is a multiplicative lattice on $X$ consisting of skeletally open maps. Our arguments follow the proof of Proposition 3.7 from \cite{vv}). Let $X_\varphi=\varphi(X)$, $\varphi\in\Phi_X$.

\textit{Claim $6.$ Any $\varphi\in\Phi_X$ is skeletal.}

Let $U\subset X$ be open in $X$ and $\varphi=\phi|X$ for some $\mathrm{e}$-admissible $\phi\in\Phi$. Because $\phi$ is
open, it suffices to show that $\displaystyle\phi(\reg(U))\cap
X_\varphi\subset\overline{\varphi(U)}^{X_\varphi}$. Suppose there
exists a point $z\in\phi(\reg(U))\cap
X_\varphi\backslash\overline{\varphi(U)}^{X_\varphi}$ and take
$V\in\mathcal B_\phi$ containing $z$ such that
$V\cap\overline{\varphi(U)}=\varnothing$ (here
$\overline{\varphi(U)}$ is the closure in $M_\phi$). Since
$\phi$ is $\reg$-admissible,
$\phi^{-1}\big(\phi\big(\overline{\reg(U_1)}\big)\big)=\overline{\reg(U_1)}$,
where $U_1=\phi^{-1}(V)\cap X$. Obviously, $U_1\cap
U=\varnothing$ and $\varphi(U_1)=V\cap X_\varphi$. Because
$\reg(U_1)\cap X$ is dense in $U_1$, we have
$\overline{\phi\big(\reg(U_1)\cap
X\big)}=\overline{\varphi(U_1)}=\overline{V\cap X_\varphi}$. Since
$\phi\big(\overline{\reg(U_1)}\big)$ is closed in $M_\phi$
(recall that $\phi$, being open, is a quotient map),
$z\in\phi(\overline{\reg(U_1)})\cap\phi(\reg(U))$ which
implies $\overline{\reg(U_1)}\cap\reg(U)\neq\varnothing$. So,
$\reg(U_1)\cap\reg(U)\neq\varnothing$, and consequently, $U\cap
U_1\neq\varnothing$. This contradiction completes the proof of the claim.

\textit{Claim $7.$ The diagonal product of any family $\{\phi_\alpha:\alpha\in A\}$ of $\reg$-admissible maps is $\reg$-admissible.}

For every $\alpha\in A$ fix a base $\mathcal B_{\phi_\alpha}$ for $M_{\phi_\alpha}$ satisfying condition $(6)$.
Let $\phi_0=\triangle_{\alpha\in A}\phi_\alpha$ and $V\in\mathcal B_{\phi_0}$, where $\mathcal B_{\phi_0}$ is the standard open base
of $M_{\phi_0}$ generated by all $\mathcal B_{\phi_\alpha}$. Then $\phi_0^{-1}(V)=\cap_{i=1}^{i=k}\phi_{\alpha(i)}^{-1}(V_i)$ for some
$V_i\in\mathcal B_{\phi_{\alpha(i)}}$. The equality $\mathrm{e}(\phi^{-1}(V)\cap X)=\cap_{i=1}^{i=k}\mathrm{e}(\phi_{\alpha(i)}^{-1}(V_i)\cap X)$ together with the facts that $\phi_{\alpha_0}$ is open and $\phi_{\alpha_0}\prec\phi_{\alpha_i}$ for each $i$, implies that
$\phi_{\alpha_0}^{-1}(\phi_{\alpha_0}(\overline{\mathrm{e}(\phi_{\alpha_0}^{-1}(V)\cap X)}))=\overline{\mathrm{e}(\phi_{\alpha_0}^{-1}(V)\cap X)}$.
Hence, $\phi_{\alpha_0}$ is $\reg$-admissible.

{\em Claim $8$. For every map $f:X\to f(X)$ there exists $\varphi\in\Phi_X$ such that
$\varphi\prec f$  and $w(X_\varphi)\leq w(f(X))$.}

\smallskip
We embed $f(X)$ in $\mathbb I^\tau$, where $\tau=w(f(X))$. Since $X$ is $C^*$-embedded in $M$, $f$ can be extended to a map
$\overline{f}:M\to\mathbb I^\tau$. Next, there exists $\phi_0\in\Phi$ with $\phi_0\prec\overline{f}$ and $w(M_{\phi_0})\leq\tau$.
We construct by induction a sequence $\{\phi_n\}_{n\geq
0}\subset\Phi$, such that for every $n\geq 0$ we have:
\begin{itemize}
\item $\phi_{n+1}\prec\phi_n$;
\item $w(M_{\phi_n})\leq\tau$;
\item $\phi_{n+1}$ satisfies condition $(6)$ for all $U\in\mathcal B_n$ with $\mathcal B_n$ being a base for $M_{\phi_n}$
of cardinality $\leq\tau$.
\end{itemize}

Suppose $\phi_n$ is already constructed. For each
$U\in\mathcal B_n$ there exists $\phi_U\in\Phi$ such that
$\phi_{U}^{-1}\big(\phi_{U}\big(\overline{\reg(\phi_{U}^{-1}(U)\cap
X)}\big)\big)=\overline{\reg(\phi_{U}^{-1}(U)\cap X)}$ and $w(M_{\phi_U})$ is countable (see
Lemma 3.6 from \cite{vv} for a similar statement). Obviously, $\phi_{n+1}=\triangle\{\phi_U:U\in\mathcal B_n\}$
 the above conditions.This completes the construction.
 It is easily seen that the diagonal product $\phi$ of all $\phi_n$ is $\reg$-admissible.  Hence,
$\varphi=\phi|X$ is as required.

Therefore, $\Phi_X=\{\phi|X:\phi\hbox{~}\mbox{is $\mathrm{e}$-admissible}\hbox{~}\}$ is indeed a multiplicative lattice on $X$ consisting of skeletal maps. Finally, according to (i), $X$ is skeletally Dugundji.
\end{proof}

\begin{cor}
Let $M$ be a product of separable metric spaces. Then every strongly $\pi$-regularly  $C^*$-embedded subset of $M$ is skeletally Dugundji.
\end{cor}

A map $f\colon X\to Y$ is said to have a $\pi$-weight $\leq\tau$ \cite{sh}, notation $\pi w(f)\leq\tau$, if there exists a family $\mathcal B$ of functionally open sets in
$X$ with $|\mathcal B|\leq\tau$ such that $\gamma=\{U\cap f^{-1}(V): U\in\mathcal B, V\in\Tee_X\}$ is a $\pi$-base for $X$ (i.e., every open subset of $X$ contains some $W\in\gamma$). Everywhere below $p(X)$ denotes the absolute of $X$ and $\theta_X\colon p(X)\to X$ is the canonical irreducible map.

\begin{thm}\label{3.3}
For any space $X$ the following are equivalent:
\begin{itemize}
\item[(i)] $X$ is skeletally Dugundji;
\item[(ii)] Every compactification of $X$ is co-absolute to a Dugundji space;
\item[(iii)] $X$ has a compactification co-absolute to a Dugundji space;
\item[(iv)] Every $C^*$-embedding of the absolute $p(X)$ in another space is strongly $\pi$-regular;
\item[(v)] $X$ has a multiplicative lattice of skeletal maps.
\end{itemize}
\end{thm}

\begin{proof}
$(i)\Rightarrow (ii)$. Since every compactification of $X$ is an irreducible image of $\beta X$, it suffices to show that $\beta X$ is co-absolute with a Dugundji space.
Let $S=\displaystyle\{X_\alpha, p^{\gamma}_\alpha, \alpha<\gamma<\tau\}$ be a factorizing inverse system with skeletal maps such that $X=\mathrm{a}-\displaystyle\lim_\leftarrow S$,
$X_0$ is second countable and $p^{\alpha+1}_\alpha$ has a metrizing kernel for all $\alpha$.
Take a second countable compactification $Y_0$ of $X_0$ and consider the inverse system $\tilde{S}=\displaystyle\{Y_\alpha, q^{\gamma}_\alpha, \alpha<\gamma<\tau\}$, where $Y_\alpha=\beta X_\alpha$, $q^{\gamma}_\alpha=\beta(p^{\gamma}_\alpha)$ for $\alpha>0$, and
$q^{\gamma}_0\colon\beta X_\gamma\to Y_0$ is the extension of $p^{\gamma}_0$. Then $\pi w(q^{\alpha+1}_\alpha)\leq\aleph_0$ for each $\alpha$  because $p^{\alpha+1}_\alpha$ has a metrizable kernel. Moreover, all projections $q^{\gamma}_\alpha$ are skeletal. Since $S$ is factorizing, $\beta X$ is the limit space of $\tilde{S}$.
Therefore, we can apply
Shapiro's result \cite[Theorem 2]{sh} to conclude that $\beta X$ is co-absolute to a Dugundji space.

$(ii)\Leftrightarrow (iii)$. If there exists a compactification $c(X)$ of $X$ which is co-absolute to a
 Dugundji space, then so is $\beta X$ as an irreducible preimage of $X$. Hence, every compactification of $X$ has this property.
 
$(ii)\Rightarrow (iv)$. Suppose $\beta X$ is co-absolute to an $AE(0)$-space $Y$ and let
$p(X)$ be $C^*$-embedded in a space $Z$. Then $\overline{p(X)}^{\beta Z}$ is homeomorphic to $\beta p(X)$. Since $\beta p(X)=p(\beta X)$, $\beta p(X)$ is the absolute of $Y$ . Consider the canonical irreducible maps $\theta_{\beta Z}: p(\beta Z)\to\beta Z$ and $\theta_Y:\beta p(X)\to Y$.
Because $Y\in AE(0)$, the restriction $\varphi= \theta_Y\circ\theta_{\beta Z}|H:H\to Y$ has a continuous extension $\phi: p(\beta Z)\to Y$ , where
$H=\theta_{\beta Z}^{-1}(\beta p(X))$. Finally, for every open $U\subset\beta p(X)$ define
$\mathrm{e}(U)=\theta_{\beta Z}^\sharp(\phi^{-1}(\theta_Y^\sharp(U)))$. It is easily seen that $\mathrm{e}:\Tee_{\beta p(X)}\to\Tee_{\beta Z}$ is a strongly $\pi$-regular operator. This implies that $p(X)$ is strongly $\pi$-regularly embedded in $Z$.

$(iii)\Rightarrow (iv)$.
Obviously, item $(iii)$ implies that every embedding of $p(\beta X)$ in any Tychonoff cube $\mathbb I^\tau$ is strongly $\pi$-regular. So, by Corollary 3.2, $p(\beta X)$ is the limit
space of a continuous inverse system $\{Z_\alpha, q^{\gamma}_\alpha, \alpha<\gamma<\tau\}$ with skeletal projections such that $Z_0$
is a metric compactum and each $q^{\alpha+1}_\alpha$ has a metrizable kernel. Then, according to \cite[Theorem 2]{sh}, $p(\beta X)$ is
co-absolute to a Dugundji space $Y$. Because $p(\beta X)$ is extremally disconnected, it is the absolute of $Y$. Since $Y$ has a multiplicative
 lattice of open maps (as an $AE(0)$-space), $\beta X$ has a multiplicative lattice $\Psi$ of skeletal maps (by Theorem 2.2).
Finally, observe that  $\Psi_X=\{\varphi|X:\varphi\in\Psi\}$ is a multiplicative lattice on $X$ consisting of skeletal maps.

$(iv)\Rightarrow (i)$. This implication follows directly from Proposition 3.1.
\end{proof}

Recall that $X$ is skeletally generated \cite{vv} if there exists a factorizing inverse
system $\displaystyle S=\{X_\alpha, p^{\beta}_\alpha, A\}$ consisting of
separable metric spaces $X_\alpha$ and surjective skeletal bonding maps such that:
(i) The index set $A$ is $\sigma$-complete (every countable chain in
$A$ has a supremum in $A$);
(ii) For every countable chain $\{\alpha_n:n\geq 1\}\subset A$ with
$\beta=\sup\{\alpha_n:n\geq 1\}$ the space $X_\beta$ is a (dense)
subset of
$\displaystyle\lim_\leftarrow\{X_{\alpha_n},p^{\alpha_{n+1}}_{\alpha_n}\}$;
(iii) $X$ is embedded in $\displaystyle\lim_\leftarrow
S$ such that $p_\alpha(X)=X_\alpha$ for each $\alpha$.

\begin{cor}
Every skeletally Dugundji space is skeletally generated.
\end{cor}

\begin{proof}
Let $\Psi$ be a multiplicative lattice for $X$ consisting of skeletal maps and $\Psi_0=\{\phi\in\Psi:w(\phi(X))\leq\aleph_0\}$.
Define $X_\phi=\phi(X)$ and equip $\Psi_0$ with the partial order $\phi_1\prec\phi_2$ iff there exists a map
$\phi^2_1\colon X_{\phi_1}\to X_{\phi_1}$ such that $\phi^2_1\circ\phi_2=\phi_1$. It is easily seen that the inverse system
$\displaystyle S=\{X_\phi, \phi^{\beta}_\alpha, \Psi_0\}$ satisfies all conditions from the definition of skeletally generated spaces.
\end{proof}

\begin{cor}
Any perfect image of a normal space possessing a multiplicative  lattice of open maps is skeletally Dugundji. In particular, any dyadic compactum is
skeletally Dugundji.
\end{cor}

\begin{proof}
Suppose $M$ is a normal space with a multiplicative lattice of open maps and let $f\colon M\to X$ be a perfect surjection. Then there exists
a closed subset $Z$ of $M$ such that the restriction $g=f|Z\colon Z\to X$ is irreducible. One can see that
$\mathrm{e}(U)=f^{-1}(g^\sharp(U))$, $U\in\Tee_Z$, defines a strongly $\pi$-regular operator $\mathrm{e}\colon\Tee_{Z}\to\Tee_{M}$. Hence, by Proposition 3.1, $Z$ is a skeletally Dugundji space. Next, Theorem 3.3 yields that $\beta Z$ is co-absolute to a Dugundji space. Finally, since
$\beta g\colon\beta Z\to\beta X$ is irreducible, $p(\beta X)=p(\beta Z)$.  Therefore, $X$ is skeletally Dugundji.
\end{proof}

\begin{cor}
Let $X$ be a skeletally Dugundji space. Then for every $C^*$-embedding of $p(X)$ in any extremally
 disconnected space $Z$ there exists a retraction from $\beta Z$ onto $p(\beta X)$.
\end{cor}

\begin{proof}
By Theorem 3.3, there exists a
strongly $\pi$-regular operator
$\mathrm{e}\colon\Tee_{p(\beta X)}\to\Tee_{\beta Z}$. Then the operator $\mathrm{\overline{e}}\colon\Tee_{p(\beta X)}\to\Tee_{\beta Z}$, $\mathrm{\overline{e}}(U)=\overline{\mathrm{e}(U)}$ has the following properties:
$\mathrm{\overline{e}}(U\cap V)=\mathrm{\overline{e}}(U)\cap\mathrm{\overline{e}}(V)$ and
$\mathrm{\overline{e}}(U)\cap p(X)=\overline{U}$ for any open sets $U,V\subset p(\beta X)$. Next, define the set-valued map
$\Phi\colon\beta Z\to 2^{p(\beta X)}$, $\Phi(z)=\bigcap\{\overline{U}:z\in\mathrm{\overline{e}}(U), U\in\Tee_{p(\beta X)}\}$ if
$z\in\bigcup\{\mathrm{\overline{e}}(U):U\in\Tee_{p(\beta X)}\}$, and $\Phi(z)=p(\beta X)$ otherwise. It is easily seen that $\Phi$ is an upper semi-continuous map with compact non-empty values and $\Phi(z)=z$ for $z\in p(\beta X)$. Finally, according to \cite{has}, $\Phi$ has a continuous selection $r\colon\beta Z\to p(\beta X)$. Obviously,  $r$ is a retraction.
\end{proof}

A combination of Corollary 3.5 and Corollary 3.6 implies a generalization of the following result \cite{b}: the absolute of any dyadic compactum is a retract of any extremally disconnected space in which it is embedded.
\begin{cor}
Any product of skeletally Dugundji spaces is skeletally Dugundji.
\end{cor}

\begin{proof}
Suppose $X =\prod_{\alpha\in A}X_\alpha$ with each $X_\alpha$ being skeletally Dugundji.
We can suppose that all $X_\alpha$ are compact spaces. Then, according
to Theorem 3.3, for each $\alpha$ there exists a strong $\pi$-regular operator
$\mathrm{e}\colon\Tee_{p(X_\alpha)}\to\Tee_{\mathbb I^{A(\alpha)}}$
 for some $A(\alpha).$ Denote by $\mathcal B$ the family of all
standard open sets $U = U_{\alpha_1}\times\ldots\times U_{\alpha_k}\times\{p(X_\alpha):\alpha\ne\alpha_i,i=1,\ldots ,k\}$ in 
$Z =\prod_{\alpha\in A} p(X_\alpha)$, and let $\reg'(U) = \reg_{\alpha_1}(U_{\alpha_1})
\times\ldots\times\reg_{\alpha_k}(U_{\alpha_k})\times\{\mathbb I^{A(\alpha)}:\alpha\ne\alpha_i,i=1,\ldots ,k\}$. Finally, the function $\reg\colon\Tee_Z\to\Tee_Y$, $\reg(W) =\bigcup\{\reg'(U):U\subset W, U\in\mathcal B\}$, where $Y =\prod_{\alpha\in A}\mathbb I^{A(\alpha)}$, is a strong $\pi$-regular
operator. Hence, by Proposition 3.1(ii), $Z$ is skeletally Dugundji. Since
the map $\theta =\prod_{\alpha\in A}\theta_\alpha\colon Z\to X$ is irreducible, $X$ and $Z$ are co-absolute.
Therefore, $X$ is skletally Dugundji.
\end{proof}

For any map $f\colon X\to Y$ let $p(f)$ denote the absolute of $f$. It is well known that $p(f)$ is a map from $p(X)\to p(Y)$ such that
$\theta_Y\circ p(f)=\theta_X\circ f$, where $p(X)$ and $p(Y)$ are the absolutes of $X$ and $Y$, respectively, and $\theta_X\colon p(X)\to X$,
$\theta_Y\colon p(Y)\to Y$ are the corresponding canonical maps. In general, a map $f$ can have many absolutes, but it has a unique one if $f$ is skeletal.

Next lemma was established in \cite[Lemma 1]{sh} in the case $X$ and $Y$ are compact, but the same arguments provide a more general version.

\begin{lem}
Let $f\colon X\to Y$ be a perfect skeletal map of a countable $\pi$-weight with $X$ and $Y$ being 0-dimensional Lindel\"{o}f p-spaces and $Y\in AE(0)$. Then there exist a space $Z$ and perfect maps $g\colon X\to Z$, $h\colon Z\to Y$ such that: $f=h\circ g$; $g$ is irreducible; $h$ is an open map having a metrizable kernel; $\dim Z=0$.
\end{lem}

\begin{thm}
The following conditions are equivalent for any \u{C}ech-complete Lindel\"{o}f p-space $X$:
\begin{itemize}
\item[(i)] $X$ is the limit space of a continuous inverse system $S=\displaystyle\{X_\alpha, \varphi^{\gamma}_\alpha, \alpha<\gamma<\tau\}$ with perfect skeletal bonding maps such that $X_0$ is a complete separable metric space and each $\varphi^{\alpha+1}_\alpha$ has a metrizable kernel;
\item[(ii)] $X$ is co-absolute with an $AE(0)$-space;
\item[(iii)] $X$ has a multiplicative lattice of perfect skeletal maps.
\end{itemize}
\end{thm}

\begin{proof}
$(i)\Rightarrow (ii)$. Following the arguments from the proof of \cite[Theorem 2]{sh}, we construct an inverse system
 $E=\displaystyle\{Y_\alpha, q^{\gamma}_\alpha, \alpha<\gamma<\tau\}$ consisting of 0-dimensional $AE(0)$-spaces $Y_\alpha$ and open perfect bonding maps such that all $q^{\alpha+1}_\alpha$ have metrizable kernels and $Y_\alpha$ are co-absolute to $X_\alpha$. Assume we already have constructed
$Y_\alpha$. So, we have perfect irreducible maps $\theta_{X_\alpha}\colon p(X_\alpha)\to X_\alpha$ and $\theta_{Y_\alpha}\colon p(X_\alpha)\to Y_\alpha$. Since $\varphi^{\alpha+1}_\alpha$ is a perfect skeletal map with a metrizable kernel, its absolute $p(\varphi^{\alpha+1}_\alpha)\colon p(X_{\alpha+1})\to p(X_\alpha)$ is  a perfect open map of countable $\pi$-weight.
Then $\theta_{Y_\alpha}\circ p(\varphi^{\alpha+1}_\alpha)\colon p(X_{\alpha+1})\to Y_\alpha$  is a perfect skeletal map of countable $\pi$-weight.
Moreover, $Y_\alpha$ and $p(X_{\alpha+1})$ are both \u{C}ech-complete Lindel\"{o}f p-spaces (recall that the last property is invariant under images and preimages of perfect maps). Hence, by Lemma 3.8, there exist a 0-dimensional space $Y_{\alpha+1}$ and perfect maps
$q^{\alpha+1}_\alpha\colon Y_{\alpha+1}\to Y_{\alpha}$, $\phi\colon p(X_{\alpha})\to Y_{\alpha+1}$ such that $q^{\alpha+1}_\alpha$ is open with a
metrizable kernel and $\phi$ is irreducible. Since $p(X_{\alpha})$ is extremally disconnected and $\phi$ is irreducible, $p(X_{\alpha})$ is the absolute of $Y_{\alpha+1}$ and $\phi$ coincides with the canonical projection $\theta_{Y_{\alpha+1}}$. Because $Y_\alpha\in AE(0)$ and $q^{\alpha+1}_\alpha$ is open and has a metrizable kernel, $Y_{\alpha+1}\in AE(0)$.

If $\alpha<\tau$ is a limit ordinal and the construction was done for all $\beta<\alpha$, we take $Y_\alpha$ to be the limit of the inverse system
$\displaystyle\{Y_\beta, q^{\gamma}_\beta, \beta<\gamma<\alpha\}$. Using again that all $Y_\beta\in AE(0)$ and $q^{\beta+1}_\beta$ are open perfect maps with metrizable kernels, we obtain that $Y_\alpha$ is a 0-dimensional $AE(0)$-space which is co-absolute with $X_\alpha$. This completes the construction. Finally, observe that the space $Y=\displaystyle\lim_\leftarrow E$ is co-absolute to $X$ and $Y\in AE(0)$.

$(ii)\Rightarrow (iii)$. Since $X$ is co-absolute to an $AE(0)$-space $Y$, $Y$ is also a \u{C}ech-complete Lindel\"{o}f p-space. Consequently,
it has a multiplicative lattice of open perfect maps. Then, by Corollary 2.4, $X$ has a multiplicative lattice of skeletal perfect maps.

$(iii)\Rightarrow (i)$ This implication follows from the arguments used in the proof of Proposition 3.1(i).
\end{proof}


\end{document}